\documentclass[10pt]{amsart}
\usepackage{amsmath,amsthm,amssymb,amsfonts, eucal, amscd, mathbbol, mathrsfs, mathabx}
\usepackage[shortlabels]{enumitem}
\usepackage{cite}
\usepackage{setspace}
\usepackage[all,cmtip]{xy}{
\usepackage{graphicx}
\usepackage{subfig}
\usepackage{fancyhdr}
\usepackage{latexsym}
\usepackage{fncylab}
\usepackage{epic}
\usepackage{ifthen}
\usepackage[bookmarks,bookmarksnumbered, plainpages=false, pdfpagelabels]{hyperref}
\usepackage{xcolor}
\hypersetup{
  colorlinks   = true, 
  urlcolor     = green, 
  linkcolor    = blue, 
  citecolor   = red 
}

\usepackage{rotating}
\usepackage{scalefnt}
\usepackage{enumitem}
\setlist{nolistsep}

\parskip = 0.2in
\parindent = 0.0in
\topmargin = 0.0in

\oddsidemargin = 0.0in
\evensidemargin = 0.0in
\textwidth = 6.0in

\newtheorem{thm}{Theorem}[section]
\newtheorem*{thm*}{Main Theorem}
\newtheorem*{thm**}{Theorem}
\newtheorem{cor}[thm]{Corollary}
\newtheorem{lem}[thm]{Lemma}
\newtheorem{prop}[thm]{Proposition}
\theoremstyle{definition}
\newtheorem{defn}[thm]{Definition}
\theoremstyle{definition}

\theoremstyle{definition}

\theoremstyle{definition}

\theoremstyle{definition}
\newtheorem{examples}[thm]{Examples}
\theoremstyle{definition}

\theoremstyle{definition}

\numberwithin{thm}{subsection}

\newcommand{\R}{\ensuremath{\mathbb{R}}}
\newcommand{\N}{\ensuremath{\mathbb{N}}} 

\newcommand{\Z}{{\Bbb Z}} 

\def\p{\partial}
\def\i{\infty}

\def\fr{\emph{fr}}
 
\def\image{\emph{im}}

\def\d{\delta}
 
\def\e{\epsilon}
\def\k{\kappa}

\def\o{\omega} 
\def\O{\Omega}

\def\H{\mathcal{H}}

\def\Span{\mathfrak{S}}

\def\XXint#1#2#3{{\setbox0=\hbox{$#1{#2#3}{\int}$}
\vcenter{\hbox{$#2#3$}}\kern-.5\wd0}}

\renewcommand{\theenumi}{(\alph{enumi})}

\makeatletter
\renewcommand{\p@enumii}{\theenumi}
\makeatother                                                

\begin{document} 
	\author[J. Harrison \& H. Pugh]{J. Harrison \\Department of Mathematics \\University of California, Berkeley  \\ \\  H. Pugh\\ Mathematics Department \\ Stony Brook University} 
	\title{Spanning via \v{C}ech cohomology}      

\begin{abstract}  
	Plateau's problem is to find a surface with minimal area spanning a given boundary. In 1960, Reifenberg and Adams developed a definition for ``span'' using \v{C}ech homology, and variants of this definition have been used ever sense. However, limitations of \v{C}ech homology resulted in the lack of a natural definition for a boundary consisting of more than one component. The authors avoided this problem in an earlier paper for codimension one surfaces using linking numbers to define spanning sets. In this paper, we show how to use \v{C}ech cohomology to provide a similar definition for all dimensions and codimensions.
\end{abstract}

\maketitle
  
\section*{Notation}
	\label{sec:notation}
	If \( X\subset \R^n \),
	\begin{itemize}
		\item \( \fr\,X \) is the frontier of \( X \);
		\item \( \bar{X} \) is the closure of \( X \);
		\item \( \mathring{X} \) is the interior of \( X \);
		\item \( \mathcal{N}(X,\e) \) is the open epsilon neighborhood of \( X \);
		\item \( \H^m(X) \) is the \( m \)-dimensional (normalized) Hausdorff measure of \( X \);
		\item If the Hausdorff dimension of \( X \) is \( m \), then the \emph{\textbf{core}} of \( X \) is the set \( X^*:= \{p \in X \,|\, \H^m(X \cap \mathcal{N}(p,r)) > 0 \text{ for all } r > 0 \} \).
	\end{itemize}

\section{Cohomological spanning condition}
\label{sub:cohomological}

	Let \( 1\leq m\leq n \) and \( A\subset \R^n \). If \( \mathcal{R} \) is a commutative ring and \( G \) is a \( \mathcal{R} \)-module, let \( H^{m-1}(A)=H^{m-1}(A;G) \) (resp. \( \tilde{H}^{m-1}(A)=\tilde{H}^{m-1}(A;G) \)) denote the \( (m-1) \)-st (resp. \( (m-1) \)-st reduced\footnote{Let us agree for notational purposes that \( \tilde{H}^0(\emptyset)=0 \) and that the inclusion \( \iota(Y,\emptyset) \) of \( \emptyset \) into any set \( Y \) induces the zero homomorphism \( \iota(Y,\emptyset)^*: \tilde{H}^0(Y)\to \tilde{H}^0(\emptyset) \).}) \v{C}ech cohomology group with coefficients in \( G \). If \( X\supset A \), and \( \iota = \iota(X,A) \) denotes the inclusion mapping of \( A \) into \( X \), let \( K^*(X,A) \) denote the complement in \( \tilde{H}^{m-1}(A) \) of the image of \( \iota^*:\tilde{H}^{m-1}(X)\to \tilde{H}^{m-1}(A) \). Call \( K^*(X,A) \) the \emph{\textbf{(algebraic) coboundary\footnote{In the spirit of Reifenberg and Adams's terminology ``algebraic boundary.''} of \( X \) with respect to \( A \).}}

	Let \( L \subset \tilde{H}^{m-1}(A)\setminus \{0\} \). We say that \( X \) is a \emph{\textbf{surface with coboundary \( \supset L \)}} if \( K^*(X,A) \supset L \); in other words, if \( L \) is disjoint from the image of \( \iota^* \).

	For example, if \( L=\emptyset \), then every \( X \supset A \) is a surface with coboundary \( \supset L \). If \( L\neq \emptyset \) and \( X \) is a surface with coboundary \( \supset L \), then \( X \) does not retract onto \( A \). If \( L= \tilde{H}^{m-1}(A)\setminus \{0\} \), then \( X \) is a surface with coboundary \( \supset L \) if and only if \( \iota^* \) is trivial on \( \tilde{H}^{m-1}(X) \).

	If \( A \) is homeomorphic to an \( (m-1) \)-sphere, \( \mathcal{R}=G=\Z \), \( L\simeq \{1,-1\} \) is the set of generators of \( \tilde{H}^{m-1}(A)\simeq \Z \) and \( X\supset A \) is compact with \( \H^m(X)<\i \), then \( X \) is a surface with coboundary \( \supset L \) if and only if \( X \) does not retract onto \( A \). This is due to a theorem of Hopf \cite{hurewicz}.

	More generally, if \( G=\mathcal{R} \) and \( A \) is an \( (m-1) \)-dimensional closed \( \mathcal{R} \)-orientable (topological) manifold, then there is a canonical choice for \( L \), denoted \( L^{\mathcal{R}}=L^\mathcal{R}(A) \): Let \( A_i, i=1,\dots,k \) denote the components of \( A \), and for each \( i \), let \( L_i \) denote the image under the natural linear embedding \( H^{m-1}(A_i)\hookrightarrow H^{m-1}(A)\simeq \oplus_i H^{m-1}(A_i) \) of the \( \mathcal{R} \)-module generators of \( H^{m-1}(A_i)\simeq \mathcal{R} \). If \( m>1 \), let \( L^\mathcal{R}=\cup_i L_i \). If \( m=1 \), define \( L^\mathcal{R} \) to be the projection of \( \cup_i L_i \) onto the reduced cohomology \( \tilde{H}^0(A) \).
	
	A primary reason for considering the set \( L^\mathcal{R} \) is the following: If \( X \) is a compact \( \mathcal{R} \)-orientable manifold with boundary \( A \), then \( X \) is a surface with coboundary \( \supset L^\mathcal{R} \) (Theorem \ref{thm:manifoldspans}.) If \( \mathcal{R}=\Z \), then \( X \) need not be orientable. In fact, if \( X \) is any compact set which can be written as the union of \( A \) and an increasing union of a sequence of compact manifolds with boundary \( X_i \), such that \( \p X_i\cup A=\p B_i \) for a sequence \( \{B_i\} \) of compact manifolds which tend to \( A \) in Hausdorff distance, then \( X \) is a surface with coboundary \( \supset L^\Z \) (Theorem \ref{thm:flatspans}.) When \( n=3, m=2 \), this is the class of surfaces \( \mathcal{G} \) found in \cite{reifenberg}.

	Another feature of \( L^\mathcal{R} \) is the following gluing result: Suppose \( A=A_1\cup\cdots\cup A_k \), where \( A_1,\dots, A_k \) are \( (m-1) \)-dimensional closed \( \mathcal{R} \)-orientable manifolds, and every non-empty intersection of the \( A_i \)'s is also a \( (m-1) \)-dimensional closed manifold. If for each \( i=1,\dots,k \), \( X_i \) is a surface with coboundary \( \supset L^\mathcal{R}(A_i) \), then \( X=\cup_i X_i \) is a surface with coboundary \( \supset L^\mathcal{R}(A) \) (Proposition \ref{prop:unionspans}.)

	If \( A \) is a \( (n-2) \)-dimensional closed oriented manifold, then by Alexander duality \( X \) is a surface with coboundary \( \supset L^\Z \) if and only if \( X \) intersects every embedding \( \gamma: \amalg_{j=1}^l S^1 \to \R^n\setminus A \), \( l\in \N \), such that the linking number \( L(\gamma, A_i) \) with some component \( A_i \) of \( A \) is \( \pm 1 \), and such that \( L(\gamma,A_j)=0 \) for \( j\neq i \).

	More generally, if \( A \) is any compact subset of \( \R^n \), we can view \( A \) as a compact subset of the \( n \)-sphere, the \( 1 \)-point compactification of \( \R^n \). Then by Alexander duality, the choice of \( L \) is equivalent to the choice of a subset \( S \) of \( \tilde{H}_{n-m-1}(S^n\setminus A) \). A compact set \( X \) is a surface with coboundary \( \supset L \) if and only if \( X \), viewed as a subset of \( S^n \), intersects the carrier of every singular chain representing an element \( S \). 

	If \( U\supset A\), let \( \Span(A,U,G,L,m) \) denote the collection of compact surfaces \( X \subset U \) such that \( X \cup A \) is a surface with coboundary \( \supset L \) (w.r.t. \( G \)) and \( \H^m(X) < \i \). 
	\newpage
	\begin{examples} \mbox{}
		\begin{enumerate}
			\item\label{threerings}
			If \( A\subset \R^3 \) is the union of three stacked circles, explicitly \( A=\{(x,y,z)\in\R^3: x^2+y^2=1, z\in\{-1,0,1\} \} \), then the surfaces \( X_1=\{(x,y,z)\in \R^3 : x^2+y^2=1, -1\leq z \leq 0 \}\cup \{(x,y,z)\in \R^3 : x^2+y^2\leq 1, z=1 \}\), \( X^2=\{(x,y,z)\in \R^3 : x^2+y^2=1, 0\leq z \leq 1 \}\cup \{(x,y,z)\in \R^3 : x^2+y^2\leq 1, z=-1 \}\) and \( X_3=\{(x,y,z)\in \R^3 : x^2+y^2=1, -1\leq z \leq 1 \} \) are all surfaces with coboundary \( \supset L^\Z \). One can replace the cylinders with catenoids, and move the circles of \( A \) up or down, in which case any of \( X_1 \), \( X_2 \) or \( X_3 \) could be an area minimizer in \( \Span(A,\R^3,L^\Z,2) \), depending on the distance between the circles of \( A \).

			\item
			If \( A\subset \R^3 \) is a standard \( 2 \)-torus given parametrically by \( x(\theta,\phi)=(R+r\cos\theta)\cos\phi \), \( y(\theta,\phi)=(R+r\cos\theta) \) and \( z(\theta,\phi)=r\sin\theta \), and \( L\subset H^1(A;\Z) \) consists of a single element, the class of the cocycle dual to a longitudinal circle \( \phi=\mathrm{const.} \), then \( X\supset A \) is a surface with coboundary \( \supset L \) if and only if \( X \) contains a longitudinal disk. If one replaces the minor radius \( r \) with a function \( r(\phi) \), then the set \( A\cup D \), where \( D \) is the longitudinal disk at the narrowest part of \( A \), will be an area minimizer in \( \Span(A,\R^3,L,2). \)
		\end{enumerate}

	\end{examples}

	This definition is the natural dual of the definition of a ``surface with boundary \( \supset L \)'' \cite{reifenberg} (see also \cite{almgrenannals}.) Recall \( X \) is a \emph{\textbf{(Reifenberg) surface with (algebraic) boundary \( \supset L \)}} if \( L \) is a subgroup of the kernel of \( \iota_*:H_{m-1}(A)\to H_{m-1}(X) \), where \( H_{m-1} \) denotes \v{C}ech homology with coefficients in some compact abelian group \( G \). Given a choice of \( G \) and \( L \), we call the collection of surfaces with boundary \( \supset L \), a \emph{\textbf{Reifenberg collection.}}  Reversing the variance has a number of advantages:
	\begin{enumerate}
		\item We permit \( G \) to be any \( \mathcal{R} \)-module, not just a compact abelian group;
		\item The collection of non-retracting surfaces in Theorem 2 of \cite{reifenberg} is achieved as a single collection, namely \( \Span(S^{m-1},\R^n,L^\Z,m) \);
		\item The sets \( X_1, X_2, X_3 \) in Example \ref{threerings} above are all surfaces with coboundary \( \supset L^\Z \), but the only Reifenberg collections containing all three correspond to the trivial subgroup \( L=\{0\} \), in which case every set \( X\supset A \) is a surface with boundary \( \supset L \). (See Proposition 5.0.1 in \cite{hpplateau}.) 
		\item There is a canonical choice of \( L \) in the case that \( A \) is an oriented compact manifold, namely the subset \( L^\Z \), and the collection \( \Span(A,\R^n,L^\Z,m) \) is well-behaved and large as described above. In particular, \( \Span(A,\R^3,L^\Z,2) \) contains Reifenberg's class \( \mathcal{G} \), and when \( A \) is a sphere, Reifenberg's class of non-retracting surfaces, \( \mathcal{G}^* \).
	\end{enumerate}
	
	One can replace the appendix of Adams in \cite{reifenberg} with results of \S \ref{sub:reifenberg_appendix}, and this together with the main body of \cite{reifenberg} implies
	
	\begin{thm}
		The minimum Hausdorff spherical measure in \( \Span(A,\R^n,G,L,m) \) is achieved, and if \( X \) is such a minimizer, then \( X^* \) is contained in the convex hull of \( A \), and contains no proper subset in \( \Span(A,\R^n,G,L,m) \).
	\end{thm}
	
	The same regularity results of Reifenberg also hold, in that \( X^* \) will be locally Euclidean \( \H^m \) almost everywhere. One can also run the new definition through \cite{almgrenannals} to achieve the minimization of elliptic integrands. 
	
\newpage
	\subsection{Some open questions} 
		\label{ssub:some_open_questions}

		\begin{enumerate}
			\item For what choice of \( L\subset H_m(X) \) is it true that if \( X \) is a surface with boundary \( \supset L \), then \( A \subset \overline{X\setminus A} \)? Same question for coboundary.
			\item On the other hand, if \( A \) is a \( (m-1) \)-dimensional \( \mathcal{R} \)-orientable compact manifold, \( A \subset \overline{X\setminus A} \), and \( X \) is minimal in the sense of Almgren, then does the algebraic boundary of \( X \) project nontrivially onto each copy of \( \mathcal{R} \) in \( H_{m-1}(A)\simeq \oplus \mathcal{R} \)? Same question for coboundary.
			\item If \( A\subset \overline{X\setminus A} \) and \( X \) is Almgren minimal, then does the algebraic boundary of \( X \) project nontrivially onto each copy of \( \mathcal{R} \) as above? Same question for coboundary.
			\item  If \( m=n-1 \) and \( A \) is a \( (m-1) \)-dimensional orientable compact manifold, the condition that \( X \) is a surface with coboundary \( \supset L^\Z \) is slightly relaxed from the definition of ``span'' using linking numbers in (\cite{plateau10} \cite{hpplateau},) since in the linking number definition, \( L^\Z \) need only be disjoint from the image of those cocycles which are Alexander dual to cycles represented by embedded circles, and not sums of such. In this vein, one can modify the definition of surface with coboundary \( \supset L \) so that \( L \) need only be disjoint from those elements of \( H^k(A) \) which extend over \( X \) as cocycles Alexander dual to cycles representable by manifolds of a given topological type. However, this definition seems very difficult to work with. For example, compare Lemma \ref{lem:6A} with Theorem 5.0.6 of \cite{hpplateau}.
			\item If \( X \) is a surface with algebraic boundary \( K \), what is the algebraic coboundary \( K^* \) of \( X \)? Does this duality \( K \mapsto K^* \) depend on \( X \)? Same question with \( K \) and \( K^* \) reversed.
			\item If one replaces sets with pairs, the definition can be repeated with relative cohomology: A pair \( (X,Y)\supset (A,B) \) is a surface with coboundary \( \supset L \) if \( L \) is disjoint from the image of \( \iota^*:H^{m-1}(X,Y)\to H^{m-1}(A,B) \). Is this definition useful for working with surfaces which partially span their boundaries?
 		\end{enumerate}

	\section{Cohomological spanning lemmas}
	\label{sub:reifenberg_appendix}

	We now produce a sequence of lemmas, many of whose statements are dual to those found in the appendix of \cite{reifenberg}. We do not assume sets are compact, unless the assumption is made explicit in the lemma. 
	 
	\begin{lem}\label{lem:1A}
		\( K^*(A,A) = \emptyset \).
	\end{lem}

	\begin{proof}
		The identity map on \( \tilde{H}^{m-1}(A) \) is surjective.
	\end{proof}

	\begin{lem}\label{lem:2A}
		If \( X \) is contractible and \( A \subset X \), then \( K^*(X,A) = \tilde{H}^{m-1}(A)\setminus \{0\} \).
	\end{lem}

	\begin{proof}
		By homotopy invariance, \( X \) has the reduced cohomology of a point.
	\end{proof}

	\begin{lem}\label{lem:3A}
		Suppose \( X\supset A \) and \( X = \cup_{i=1}^N X_i \) where the \( X_i \) are disjoint, closed, and contractible. If \( m>1 \), then \( K^*(X,A) = H^{m-1}(A)\setminus \{0\} \). 
	\end{lem}

	\begin{proof}
		Let \( A_i=A\cap X_i \). By E-S Ch. I Thm. 13.2c, \( H^{m-1}(X) \simeq \oplus_{i=1}^N H^{m-1}(X_i) \) and \( H^{m-1}(A)\simeq \oplus_{i=1}^N (A_i) \). Moreover, the square
		\begin{align}
			\xymatrix{H^{m-1}(X) \ar@{->}[r]^-{\simeq} \ar@{->}[d]^{\iota(X,A)^*}  & \oplus H^{m-1}(X_i) \ar@{->}[d]^{\oplus \iota(X_i,A_i)^*} \\
			\tilde{H}^{m-1}(A)  \ar@{->}[r]^-{\simeq} &  \oplus H^{m-1}(A_i) }	
		\end{align}
		commutes since it does so for each summand of \( \oplus H^{m-1}(A_i) \). We may then apply Lemma \ref{lem:2A}.
	\end{proof}

	\begin{lem}\label{lem:6A}
		Suppose \( g:(X,A) \to (Y,B) \) is continuous. Let \( L_A \subset \tilde{H}^{m-1}(A)\setminus \{0\} \) and \( L_B = (g|_A^*)^{-1}(L_A) \). If \( X \) is a surface with coboundary \( \supset L_A \), then \( Y \) is a surface with coboundary \( \supset L_B \).
	\end{lem}

	\begin{proof}
		The proof is evident from the commutativity of the following square:
		\begin{align}
			\xymatrix{\tilde{H}^{m-1}(X) \ar@{<-}[r]^{g^*} \ar@{->}[d]^{\iota(X,A)^*}  & \tilde{H}^{m-1}(Y) \ar@{->}[d]^{\iota(Y,B)^*} \\
			\tilde{H}^{m-1}(A)  \ar@{<-}[r]^{(g|_A)^*} & \tilde{H}^{m-1}(B). }	
		\end{align}
	\end{proof}

	\begin{defn}
	 	\label{def:competitor}
	 	In particular, suppose a continuous map \( g:U \to U \) is the identity on \( A\subset U \). If \( X \in \Span(A,U,L,m) \), then \( g(X) \in \Span(A,U,L,m) \). The set \( g(X) \) is called a \emph{\textbf{competitor of \( X \) in \( U \)}}.
	\end{defn}

	\begin{lem}\label{lem:7A}
		Suppose \( X \) is a surface with coboundary \( \supset L \). If \( X \subset Y \), then \( Y \) is also a surface with coboundary \( \supset L \).
	\end{lem}

	\begin{proof}
		The inclusion \( A\hookrightarrow Y \) factors through \( X \), so the image of \( \iota(Y,A)^* \) is contained in the image of \( \iota(X,A)^* \).
	\end{proof}

	\begin{lem}\label{lem:8A}
		Let \( m=n \) and suppose \( A \) is the unit sphere in \( \R^n \). If \( X\supset A \) contains the closed unit ball, then \( K^*(X,A) = \tilde{H}^{n-1}(A)\setminus \{0\} \). If \( X\supset A \) does not contain the closed unit ball, then \( K^*(X, A) = \emptyset \).
	\end{lem}

	\begin{proof}
		If \( X \) contains the closed unit ball \( B \), then Lemmas \ref{lem:2A} and \ref{lem:7A} prove the first statement. If \( X \) does not contain \( B \), then \( A \) is a retract of \( X \), and so \( \iota(X,A)^* \) must be surjective. 
	\end{proof}
 
	\begin{lem}\label{lem:10A} 
		Suppose \( f:I \times Y \to X \) is a continuous map. Set \( A_0 = f(\{0\} \times Y), \quad A_1 = f(\{1\} \times Y), \) and \( A = A_0 \cup A_1 \). Write \( f_0:= f\lfloor_{\{0\} \times Y} \) and \( f_1 :=  f\lfloor_{\{1\} \times Y} \). Suppose \( f_0 \) is a homeomorphism from \( \{0\}\times Y \) to \( A_0 \), and that we are given a subset \( L_0 \subset \tilde{H}^{m-1}(A_0)\setminus \{0\} \). Then there exists a subset \( L_1 \subset \tilde{H}^{m-1}(A_1)\setminus\{0\} \) satisfying two properties: \[ K^*(X,A) \cup (\iota(A,A_0)^*)^{-1}(L_0)  = K^*(X,A) \cup (\iota(A,A_1)^*)^{-1}(L_1) \] and if \( X \) is a surface with coboundary \( \supset L_0 \), then \( X \) is a surface with coboundary \( \supset L_1 \). 
	\end{lem}

	\begin{proof} 
		Define \( g: \{0\} \times Y \to \{1\} \times Y \) where \( g(0,y) = (1,y) \). Let \( L_0' := f_0^*(L_0) \subset \tilde{H}^{m-1}(\{0\} \times Y) \) and \( L_1':= (g^*)^{-1}(L_0') \in \tilde{H}^{m-1}(\{1\} \times Y) \). Finally define \( L_1 = (f_1^*)^{-1}(L_1') \).

		Let \( h \in (\iota(A,A_0)^*)^{-1}(L_0) \) and suppose \( h\notin K^*(X,A) \). That is, suppose \( h=\iota(X,A)^*(x) \) for some \( x\in  \tilde{H}(X) \). We want to show \( \iota(A,A_1)^*(h) \in L_1 \). That is, \( f_1^*\iota(A,A_1)^*(h) \in  L_1' \). In other words, \(g^* f_1^*\iota(A,A_1)^*(h) \in L_0' \), or \[(f_0^*)^{-1}g^* f_1^*\iota(A,A_1)^*\iota(X,A)^*(x) \in  L_0 .\] Since we have assumed \( h=\iota(X,A)^*(x)\in  (\iota(A,A_0)^*)^{-1}(L_0) \), it suffices to show \[ \iota(A,A_0)^*\iota(X,A)^*(x) = (f_0^*)^{-1}g^* f_1^*\iota(A,A_1)^*\iota(X,A)^*(x), \] which is verified by the fact that \( \iota(X,A_0) f_0 \) and \( \iota(X,A_1) f_1 g \) are homotopic. The other containment is proved in a similar manner.
		
 		For the last assertion, suppose \( h_1 \in L_1 \) and \( h_1 =  \iota(X, A_1)^*(x) \) for some \( x \in  \tilde{H}^{m-1}(X) \).  By definition of \( L_1 \), we know  \( h_0 = (f_0^*)^{-1}g^* f_1^*(h_1) \in L_0 \).  The inclusion map \( \iota(X,A_0) \) is homotopic to \( \iota(X,A_1) \circ f_1\circ g \circ f_0^{-1} \) via the homotopy \( \iota(X,A_t) \circ f_t \circ g_t \circ f_0^{-1} \) where \( A_t = f(\{t\} \times Y) \), \( f_t:= f\lfloor_{\{t\} \times Y} \) and \( g_t(0,y) = (t,y) \). Thus \( \iota(X,A_0)^* = (\iota(X,A_1) \circ f_1 \circ g \circ f_0^{-1})^* = (f_0^*)^{-1}g^* f_1^*\iota(X,A_1)^* \). 
Then \( h_0 = (f_0^*)^{-1}g^* f_1^*(h_1) = (f_0^*)^{-1}g^* f_1^*\iota(X, A_1)^*(x) = \iota(X,A_0)^*(x) \), contradicting our assumption that \( X \) is a surface with coboundary \( \supset L_0 \).
	\end{proof}
	
	It follows from Lemma \ref{lem:7A} that if \( Z \) is a surface with coboundary \( \supset L_0 \), then \( Z \cup f(I \times Y) \) is a surface with coboundary \( \supset L_1 \). 
	
	\begin{lem}\label{lem:11A}
		Suppose \( X = \cup_{r=1}^N X_r \), \( A \subset X \), and \( A_r \subset X_r \) for each \( r \). Let \( B = A \cup_r A_r \). For each \( r \), let \( L_r \subset \tilde{H}^{m-1}(A_r)\setminus \{0\} \) and suppose \( X_r \) is a surface with coboundary \( \supset L_r \). Suppose \( L \subset \tilde{H}^{m-1}(A)\setminus \{0\} \) satisfies 
		\begin{equation}\label{eq:E}
	 		(\iota(B,A)^*)^{-1}(L) \subset \cup_r (\iota(B,A_r)^*)^{-1}(L_r).
		\end{equation}
		Then \( X \) is a surface with coboundary \( \supset L \).
	\end{lem}

	\begin{proof}
		Let \( k\in \tilde{H}^{m-1}(X) \). Then \( \iota(X,A)^* k = \iota(B,A)^* i(X,B)^* k \). By assumption, it suffices to show that \( \iota(X,B)^* k \) is not contained in \( (\iota(B, A_r)^*)^{-1} (L_r) \) for each \( r \). In other words, it suffices to show that \( \iota(X,A_r)^* k \) is not contained in \( L_r \), or equivalently, \( \iota(X_r,A_r)^*\iota(X,X_r)^* k \) is not contained in \( L_r \). Indeed, this is true since the image of \( \iota(X_r,A_r)^* \) is disjoint from \( L_r \) by assumption.
	\end{proof}

	\begin{lem}\label{lem:12A}
	 	Under the same assumptions of Lemma \ref{lem:11A}, suppose further that \( X_r \) and \( A_r \) are compact for each \( r \), that \( A \cap X_r \subset A_r \) for each \( r \), and that \( X_r \cap X_s = A_r \cap A_s \) for \( r \ne s \). Then \[ K^*(X,A) = \{x\in \tilde{H}^{m-1}(A): (\iota(B,A)^*)^{-1}(x)\subset \cup_r (\iota(B,A_r)^*)^{-1}(K^*(X_r,A_r))\}. \]
	 \end{lem} 

	\begin{proof}
		By Lemma \ref{lem:11A}, \( \{x\in \tilde{H}^{m-1}(A): (\iota(B,A)^*)^{-1}(x)\subset \cup_r (\iota(B,A_r)^*)^{-1}(K^*(X_r,A_r))\}\subset K^*(X,A) \). To show the reverse inclusion, we chase the diagram below. The unlabeled maps are given by inclusions, the rows and column are exact (E-S Ch. I Thm. 8.6c,) and the isomorphism is due to excision: Assuming \( N=2 \), the isomorphism follows from E-S Ch. I Thm. 14.2c and Ch. X Thm. 5.4. The general case follows from induction on \( N \). The triangle commutes by functoriality, the top ``square'' commutes because \( \delta \) is a natural transformation, and the bottom square commutes because it does so on each summand of \( \oplus H^m(X_r,A_r). \) Thus, the diagram commutes.

		Let \( x \in K^*(X,A) \) and suppose \( p \in (\iota(B,A)^*)^{-1}(x) \). Suppose there is no \( r \) such that \( \iota(B,A_r)^*(p) \in K^*(X_r,A_r) \). Then \( y = \oplus (\iota(B,A_r)^*)(p) \in \image \oplus \iota(X_r,A_r)^* \). Since the bottom row of the diagram is exact, \( (\oplus \d)(y) = 0 \), hence \( \d p = 0 \), hence \( \d x=0 \). But the left column is exact, and this gives a contradiction, since by assumption \( x \) is not in the image of \( \iota(X,A)^* \).
		\begin{align}
			\xymatrix{&H^m(X,A) \ar@{<-}[d]^{\d}  \ar@{<-}[rrdd] &\quad & \quad  \\ &\tilde{H}^{m-1}(A) \ar@{<-}[d]  \ar@{<-}[dr]  &\quad & \quad \\
			&\tilde{H}^{m-1}(X)  \ar@{->}[r]   & \tilde{H}^{m-1}(B) \ar@{->}[d] \ar@{->}[r]^{\d} & H^m(X,B) \ar@{->}[d]^{\cong}\\
			&\oplus \tilde{H}^{m-1}(X_r)\ar@{->}[r]  &\oplus \tilde{H}^{m-1}(A_r) \ar@{->}[r]^{\oplus \d} &\oplus H^m(X_r,A_r).}
		\end{align}
	\end{proof}

	\begin{lem}\label{lem:13}
		Suppose \( A, X \) and \( C \) are compact, with \( X \in \Span(A,\R^n,L,m) \) and \( \mathring{C}\cap A=\emptyset \). If \( Y\supset X\cap \fr\, C \) is a surface with coboundary \( \supset K^*(X\cap C, X\cap \fr\, C) \), then \( (X\setminus \mathring{C}) \cup Y \) is a surface with coboundary \( \supset L \).
	\end{lem}

	\begin{proof}
		Let \( X_1 = X\cap C \), \( A_1=X\cap \fr\, C \), \( X_2=X\setminus \mathring{C} \), and \( A_2=A\cup A_1 \). Let \( L_1=K^*(X_1,A_1) \) and \( L_2=K^*(X_2, A_2) \). By Lemma \ref{lem:12A}, \[ K^*(X,A)= \left\{x\in \tilde{H}^{m-1}(A): (\iota(A_2,A)^*)^{-1}(x)\subset \left((\iota(A_2,A_1)^*)^{-1} L_1\right)\cup L_2\right\}. \] Now apply Lemma \ref{lem:11A}, using the set \( Y \) in place of \( X_1 \). The result follows, since \( L\subset K^*(X,A) \).
	\end{proof}

	\begin{lem}\label{lem:13A}
		Suppose \( A = A_1 \cup A_2 \) where \( A_1 \) and \( A_2 \) are compact. Let \( D = A_1 \cap A_2 \) and suppose \( B\supset D \) is compact. Let \( m\geq 2 \). Suppose the homomorphism \( \iota(B,D)^*: \tilde{H}^{m-2}(B)\to \tilde{H}^{m-2}(D) \) is zero. Then \[ (\iota(A\cup B,A)^*)^{-1} (H^{m-1}(A)\setminus \{0\}) \subset \cup_{i=1,2}(\iota(A\cup B, A_i\cup B)^*)^{-1} (H^{m-1} (A_i\cup B )\setminus \{0\}). \]
	 \end{lem}

	\begin{proof} 
		Suppose \( D \) is non-empty. The map \( (A,A_1,A_2) \to (A\cup B, A_1\cup B,A_2\cup B) \) is a map of compact, and hence proper triads (E-S Ch. X Thm 5.4,) and thus carries the reduced Mayer-Vietoris sequence of the second into the first (E-S 15.4c.) Chase the resulting commutative diagram, observing that \( \cup_{i=1,2}(\iota(A\cup B, A_i\cup B)^*)^{-1} (H^{m-1} (A_i\cup B )\setminus \{0\}) = (\iota_1^*,\iota_2^*)^{-1}(H^{m-1}(A_1\cup B)\oplus H^{m-1}(A_2\cup B)\setminus \{0\}) \):
		\begin{align}
			\xymatrixcolsep{3pc}\xymatrix{\tilde{H}^{m-2}(B) \ar@{->}[d]^0 \ar@{->}[r]^\Delta &H^{m-1}(A\cup B) \ar@{->}[d] \ar@{->}[r]^-{(\iota_1^*,\iota_2^*)} &H^{m-1}(A_1\cup B)\oplus H^{m-1}(A_2\cup B) \ar@{->}[d] \\
			\tilde{H}^{m-2}(D) \ar@{->}[r]^\Delta &H^{m-1}(A) \ar@{->}[r] &H^{m-1}(A_1)\oplus H^{m-1}(A_2). }
		\end{align}
		If \( D \) is empty, then we still have the right hand commuting square, and the map out of \( H^{m-1}(A) \) is an isomorphism, and in particular, injective. This proves the lemma.
	\end{proof}

	\begin{lem}\label{lem:15A}
		Suppose \( A = \cup_{r=0}^N A_r \), where each \( A_r \) is compact. Let \( D_r = A_0 \cap A_r, \quad 1 \le r \le N \) and suppose \( A_r \cap A_s = \emptyset \), \( 1 \le r < s \le N \). Let \( m\geq 2 \). For each \( 1\le r\le N \), suppose \( B_r\supset D_r \) is compact and that the homomorphism \( \iota(B_r,D_r)^*: \tilde{H}^{m-2}(B_r) \to \tilde{H}^{m-2}(D_r) \) is zero. Furthermore, suppose that the intersection \( A_r\cap B_s \) is empty for all \( 1\leq r<s\leq N \). Let
		\begin{align*}
			&C = A \cup_{r=1}^N B_r, \\
			&C_0 = A_0 \cup_{r=1}^N B_r, \,\, \mathrm{ and}\\
			&C_r = A_r \cup B_r, \, 1 \le r \le N.
		\end{align*}
		Then \[ (\iota(C,A)^*)^{-1} (H^{m-1}(A)\setminus \{0\}) \subset \cup_{r=0}^N (\iota(C, C_r )^*)^{-1} (H^{m-1} (C_r)\setminus \{0\}). \]
	\end{lem}

	\begin{proof} 
		For \( 0\leq k\leq N-1 \), let \( E_k=A_0\cup\cdots\cup A_k\cup B_{k+1}\cup\cdots\cup B_N \), and Let \( E_N=A \). For \( 1\leq k\leq N \), we may apply Lemma \ref{lem:13A} to the sets \( ``A_1"=A_k \), \( ``A_2"=A_0\cup\cdots\cup A_{k-1}\cup B_{k+1}\cup\cdots\cup B_N \), and \( ``B"=B_k \), since the assumption \( A_k\cap B_j=\emptyset \) for all \( 1\leq k<j\leq N \) guarantees that \( ``A_1"\cap ``A_2"=D_k \). The following inclusion therefore holds for all \( 1\leq k\leq N \):
		\begin{align*}
			(\iota(E_k\cup B_k, E_k)^*)^{-1}&(H^{m-1}(E_k)\setminus\{0\})\subset\\
			 (\iota(E_k\cup B_k, E_{k-1})^*)^{-1}(H^{m-1}(E_{k-1})\setminus \{0\})&\cup (\iota(E_k\cup B_k, C_k)^*)^{-1}(H^{m-1}(C_k)\setminus \{0\}).
		\end{align*}
		Taking the inverse image in \( H^{m-1}(C) \) of the above sets by the map \( \iota(C, E_k\cup B_k)^* \), this yields
		\begin{align*}
			(\iota(C, E_k)^*)^{-1}&(H^{m-1}(E_k)\setminus\{0\})\subset\\
			(\iota(C, E_{k-1})^*)^{-1}(H^{m-1}(E_{k-1})\setminus \{0\})&\cup (\iota(C, C_k)^*)^{-1}(H^{m-1}(C_k)\setminus \{0\}).
		\end{align*}
		The result follows from downward induction on \( k \) starting at \( k=N \), since \( E_N=A \) and \( E_0=C_0 \).
	\end{proof}

	\begin{lem}\label{lem:16A}
		Suppose \( A = \cup_{r=0}^N A_r \) where each \( A_r \) is compact. Let \( D_r = A_{r-1}\cap A_r,\quad 1 \le r \le N \) and suppose \( A_r \cap A_s = \emptyset \) if \( |r-s| > 1 \). Let\footnote{Note the strict inequality.} \( m>2 \). For each \( 1\le r\le N \), suppose \( B_r\supset D_r \) is compact and that the homomorphism \( \iota(B_r,D_r)^*: H^{m-2}(B_r) \to H^{m-2}(D_r) \) is zero. Let \( B_0 = B_{N+1} = \emptyset \), and suppose further that \( B_r \cap A_{r-1}=D_r \) and \( B_r \cap A_s=\emptyset \) for all \( 1\leq r \leq N \) and \( 0\leq s < r-1 \). Let
		\begin{align*}
			&C = A \cup_{r=1}^N B_r,\\
			&C_{-1}= \cup_{r=1}^N B_r, \,\, \mathrm{ and}\\
			&C_r = B_r \cup A_r \cup B_{r+1},\,0\le r\le N.
		\end{align*}
		Then \[ (\iota(C,A)^*)^{-1}(H^{m-1}(A)\setminus \{0\}) \subset \cup_{r=-1}^N(\iota(C,C_r)^*)^{-1}(H^{m-1}(C_r)\setminus \{0\}). \]

		Furthermore, if \( B_r\cap B_s=\emptyset \) for all \( r\neq s \), then \[ (\iota(C,A)^*)^{-1}(H^{m-1}(A)\setminus \{0\}) \subset \cup_{r=0}^N(\iota(C,C_r)^*)^{-1}(H^{m-1}(C_r)\setminus \{0\}). \]
	\end{lem}

	\begin{proof}
		For \( 0\leq k\leq N \), let \( E_k=A_0\cup\cdots\cup A_k\cup B_{k+1}\cup\cdots\cup B_{N+1} \). Let \( E_{-1}=C_{-1} \) and \( D_0=\emptyset \). For \( 1\leq k\leq N \), let us apply Lemma \ref{lem:13A} to the sets \( ``A_1"=A_k \), \( ``A_2"=A_0\cup\cdots\cup A_{k-1}\cup B_{k+1}\cup\cdots\cup B_{N+1} \), and \( ``B"=B_k\cup B_{k+1} \). For \( k=0 \), use \( ``A_1"=A_0 \), \( ``A_2"=C_{-1} \), and \( ``B"=B_1 \). We may do so, because our assumption on the intersections \( B_r\cap A_s \) imply \( ``A_1"\cap ``A_2"=D_k\cup D_{k+1} \). This union being disjoint, the homomorphism \( \iota(``B",``D")^*=\iota(B_k\cup B_{k+1}, D_k\cup D_{k+1})^* \) is given by the direct sum \( \iota(B_k\cup B_{k+1}, D_k)^*\oplus \iota(B_k\cup B_{k+1},D_{k+1})^* \), both of which are zero. As in the proof of Lemma \ref{lem:15A}, the following inclusion therefore holds for all \( 0\leq k\leq N \):
		\begin{align*}
			(\iota(C, E_k)^*)^{-1}&(H^{m-1}(E_k)\setminus\{0\})\subset\\
			(\iota(C, E_{k-1})^*)^{-1}(H^{m-1}(E_{k-1})\setminus \{0\})&\cup (\iota(C, C_k)^*)^{-1}(H^{m-1}(C_k)\setminus \{0\}).
		\end{align*}
		This gives the first conclusion. If \( B_r\cap B_s=\emptyset \) for all \( r\neq s \), then by additivity, \[ (\iota(C, C_{-1})^*)^{-1}(H^{m-1}(C_{-1})\setminus \{0\})=\cup_r (\iota(C,B_r)^*)^{-1}(H^{m-1}(B_r)\setminus \{0\}).\] For each \( r \), we have \( (\iota(C,B_r)^*)^{-1}(H^{m-1}(B_r)\setminus \{0\})\subset (\iota(C,C_r)^*)^{-1}(H^{m-1}(C_r)\setminus \{0\}) \) by functoriality, thus giving the second conclusion.
	\end{proof}

	\begin{lem}\label{lem:17Apre}
		If the topological dimension \( \mathrm{dim}(A) \) of \( A \) is \( \le m-2 \), then \( H^{m-1}(A) = 0 \).
	\end{lem}

	\begin{proof}
		if \( m=1 \), the result is trivial. if \( m>1 \), then every open cover \( U \) admits a refinement \( V \) of order \( \le m-2 \) (\cite{hurewicz} Theorem V 1.) The nerve \( N \) of \( V \) is then a simplicial complex of dimension \( \le m-2 \), and so if \( x\in H^{m-1}(A) \) is represented by a simplicial cochain on the nerve of \( U \), it must pull back to the zero cochain on \( N \). Thus, \( x=0 \).
	\end{proof}

	\begin{lem}\label{lem:17A}
		 If \( \H^{m-1}(A) = 0 \), then \( H^{m-1}(A) = 0 \). 
	\end{lem}

	\begin{proof}
		If \( m=1 \), the result is trivial. If \( m>1 \), then \( \mathrm{dim}(A) \le m-2 \) by \cite{hurewicz} Theorem VII 3.
	\end{proof}

	\begin{lem}\label{lem:21B}
		Suppose \( X \) is compact and \( X=\varprojlim X_i \), where \( \{X_i\} \) is a system of compact surfaces with coboundary \( \supset L \), directed under inclusion. Then \( X \) is a surface with coboundary \( \supseteq L \).
	\end{lem}

	\begin{proof}
		By continuity of \v{C}ech cohomology, the obvious map \( \varinjlim \tilde{H}^{m-1}(X_i)\to \tilde{H}^{m-1}(X) \) is an isomorphism, and in particular a surjection, so the image of \( \iota(X,A)^* \) is the union of the images of \( \iota(X_i,A)^* \).
	\end{proof}

	\begin{lem}\label{lem:21C}
		If \( \{X_i\} \) is a sequence of compact surfaces with coboundary \( \supset L \) and \( X_i \to X \) in the Hausdorff metric, then \( X \) is a surface with coboundary \( \supset L \).
	\end{lem}

	\begin{proof}
		By Lemma \ref{lem:7A}, the sets \( Y_j = X \cup \cup_{i=j}^\i X_i \) satisfy the conditions of Lemma \ref{lem:21B}.
	\end{proof}

	\begin{lem}\label{lem:corespans}
		If \( (X,A) \) is compact and \( X \) is a surface with coboundary \( \supseteq L \), then \( X^*\cup A \) is a surface with coboundary \( \supseteq L \).
	\end{lem}

	\begin{proof}
		The inclusion of \( A \) into \( X \) factors through \( X^*\cup A \), so it suffices to show \( \iota(X,X^*\cup A)^*:\tilde{H}^{m-1}(X)\to \tilde{H}^{m-1}(X^*\cup A) \) is surjective. For \( \e>0 \), let \( X_\e=X\cap \O(X^*\cup A, \e) \). Since \( \H^m(X\setminus X^*)=0 \), it follows from \cite{hurewicz} Theorem VII 3 that \( \mathrm{dim}(X\setminus X^*) \leq m-1 \). Since \( X\setminus X_\e \) is compact, we may cover \( X\setminus X_\e \) by a finite number of open subsets \( U_i \) of \( X \), \( i=1,\dots,n \), such that for each \( i \), \( U_i\subset \O(p_i,\e/2) \) for some \( p_i\in X\setminus X^* \), and \( \mathrm{dim}(\p U_i)\leq m-2 \). Define \( B_\e=\overline{\cup_{i=1}^N U_i} \) and \( C_\e=X\setminus B_\e \). Then \( B_\e \) and \( C_\e \) are compact, \( B_\e\subset X\setminus X^* \), and \( X^*\cup A\subset C_\e\subset X_\e \). Furthermore, since \( B_\e\cap C_\e=\p(B_\e)\subset \cup{i=1}^N \p U_i \), it follows from \cite{hurewicz} Theorem III 1 that \( \mathrm{dim}(B_\e\cap C_\e)\leq m-2 \). By Lemma \ref{lem:17Apre}, \( H^{m-1}(B_\e\cap C_\e)=0 \). The Mayer-Vietoris sequence applied to the compact triad \( (X, B_\e, C_\e) \) thus implies that \( \iota(X,C_\e)^*:H^{m-1}(X)\to H^{m-1}(C_\e) \) is surjective. Finally, since \( X^*\cup A=\varprojlim C_\e \), the result follows from the continuity of \v{C}ech cohomology.
	\end{proof}

	In fact, the above proof shows that if \( Y\subset X \) is compact and contains \( A \), and \( \mathrm{dim}(X\setminus Y)\leq m-1 \), then \( Y \) is a surface with coboundary \( \supset L \).

	\section{Results specific to \( L^\mathcal{R} \)}

		\begin{thm}
			\label{thm:manifoldspans}
			Suppose \( A \) is an \( (m-1) \)-dimensional closed \( \mathcal{R} \)-orientable manifold and \( X \) is a compact \( \mathcal{R} \)-orientable manifold with boundary \( A \), then \( X \) is a surface with coboundary \( \supset L^\mathcal{R} \).
		\end{thm}
		\begin{proof}
			The result is obvious if \( m=1 \). Let \( m>1 \). Since \( A \) and \( X \) have the homotopy type of \( CW \)-complexes (\cite{hatcher} Cor A.12,) we may treat the \v{C}ech cohomology groups involved in the definition of ``surface with coboundary'' as singular cohomology groups, since the two theories are naturally isomorphic. We proceed by contradiction. Suppose there exists \( \phi\in L^\mathcal{R} \) with \( \iota(X,A)^*(\omega)=\phi \) for some \( \omega\in H^{m-1}(X) \). Writing \( A=\cup A_i \) where the \( A_i \)'s are the connected components of \( A \), this means that there exists \( j \) such that \( \iota(X,A_j)^*\omega \) is a generator of \( H^{m-1}(A_j) \), and \( \iota(X,A_i)^*\omega=0 \) for all \( i\neq j \). 

			Since \( X \) is \( \mathcal{R} \)-orientable, let \( \eta\in H_m(X,A) \) be a fundamental class (see \cite{hatcher} p.253.) Then \( \nu=\p\eta\in H_{m-1}(A) \) is a fundamental class for \( A \) (\cite{hatcher} p.260,) and by exactness, \( \iota(X,A)_*\nu=0 \). Write \( \nu=\sum \iota(A,A_i)_* \nu_i \). Then \( \nu_i \) is a fundamental class for \( A_i \) for each \( i \). We have
			\begin{align*}
				0&=\o(\iota(X,A)_*\nu)\\
				&=\iota(X,A)^*\o(\nu)\\
				&=\phi(\nu)\\
				&=\sum\phi\left(\iota(A,A_i)_*\nu_i\right)\\
				&=\sum \iota(A,A_i)^*\phi (\nu_i)\\
				&=\iota(A,A_j)^*\phi (\nu_j)\\
				&\neq 0
			\end{align*}
where the last line follows from Poincar\'e duality (\cite{hatcher} Thm 3.30.)
		\end{proof}

		By reducing mod 2, the universal coefficient theorem gives the following corollary:

		\begin{cor}
			\label{cor:nonorientablespans}
			If \( A \) is an \( (m-1) \)-dimensional closed orientable manifold and \( X \) is a compact manifold with boundary \( A \), then \( X \) is a surface with coboundary \( \supset L^\Z \).
		\end{cor}

		\begin{proof}
			It follows from Theorem \ref{thm:manifoldspans} that \( X \) is a surface with coboundary \( \supset L^{\Z/2\Z} \). Suppose there exists \( \omega\in H^{m-1}(X) \) and \( j \) such that \( \iota(X,A_j)^*\omega \) is a generator of \( H^{m-1}(A_j) \), and \( \iota(X,A_i)^*\omega=0 \) for all \( i\neq j \). The cohomology class \( \omega \) gives a homomorphism \( f_\o: H_{m-1}(X;\Z)\to \Z \), and by composing with the reduction map \( \Z \to \Z/2\Z \), a homomorphism \( \tilde{f}_\o:H_{m-1}(X;\Z)\to \Z/2\Z \). Since
			\[
			H^{m-1}(X;\Z/2\Z)\to \mathrm{Hom}(H_{m-1}(X;\Z),\Z/2\Z)\to 0
			\]
			is exact, the map \( \tilde{f}_\o \) lifts to a cohomology class \( \tilde{\o}\in H^{m-1}(X;\Z/2\Z) \). Since \( [A_i]_{\Z/2\Z} \), the \( \Z/2\Z \) fundamental class of \( A_i \), is the reduction mod-\( 2 \) of \( [A_i]_\Z \), the \( \Z \) fundamental class of \( A_i \), it follows from naturality of the universal coefficient theorem and the tensor-hom adjunction that
			\begin{align*}
				\iota(X,A_i)^*\tilde{\o}\left([A_i]_{\Z/2\Z}\right)&=\iota(X,A_i)^*\tilde{f}_\o\left([A_i]_\Z\right)\\
				&=\tilde{f}_\o\left(\iota(X,A_i)_*[A_i]_\Z\right)\\
				&=\o\left(\iota(X,A_i)_*[A_i]_\Z\right)\textrm{ mod }2,
			\end{align*}
			which equals \( 1 \) if \( i=j \) and \( 0 \) otherwise. In other words, \( \iota(X,A)^*\tilde{\o}\in L^{\Z/2\Z} \), giving a contradiction.
		\end{proof}

		\begin{thm}
			\label{thm:flatspans}
			Suppose \( A \) is an \( (m-1) \)-dimensional closed orientable manifold. Suppose \( X \) is a compact set which can be written in the form \( X=A\cup\cup_i X_i \), where \( \cup_i X_i \) is a increasing union of a sequence \( \{X_i\} \) of compact manifolds with boundary, such that for each \( i \), \( \p X_i\cup A \) is the boundary of a compact manifold with boundary \( B_i \), and such that \( B_i\to A \) in the Hausdorff metric. Then \( X \) is a surface with coboundary \( \supset L^\Z \).
		\end{thm}

		\begin{proof}
			Writing \( C_N=X\cup \cup_{i=N}^\infty B_i \), it suffices to show, by Lemma \ref{lem:21B}, that for all \( N \), the set \( C_N \) is a surface with coboundary \( \supset L^\Z \). Since \( A\subset X_N\cup B_N\subset C_N \), it suffices to show, by Lemma \ref{lem:7A}, that \( X_N\cup B_N \) is a surface with coboundary \( \supset L^\Z \). Indeed it is, since the compact manifold formed by gluing \( X_N \) and \( B_N \) along their common boundary \( \p X_N \) is a manifold with boundary \( A \), and is thus a surface with coboundary \( \supset L^\Z \) by Corollary \ref{cor:nonorientablespans}. The set \( X_N\cup B_N \) is the continuous image of this manifold, and therefore is a surface with coboundary \( \supset L^\Z \) by Lemma \ref{lem:6A}.
		\end{proof}

		\begin{prop}
			\label{prop:unionspans}	
			Suppose \( A=A_1\cup\cdots\cup A_k \), where \( A_1,\dots, A_k \) are \( (m-1) \)-dimensional closed \( \mathcal{R} \)-orientable manifolds. Suppose that every non-empty intersection of the \( A_i \)'s is also a \( (m-1) \)-dimensional closed manifold, or equivalently that every component of \( A \) is contained in some \( A_i \). Then \( A \) is a \( \mathcal{R} \)-orientable closed manifold, and if \( X_i \) is a surface with coboundary \( \supset L^\mathcal{R}(A_i) \), \( i=1,\dots\k \), then \( X=\cup_i X_i \) is a surface with coboundary \( \supset L^\mathcal{R}(A) \).
		\end{prop}

		\begin{proof}
			The equivalence of the assumptions in the second sentence is a consequence of Brouwer's invariance of domain theorem \cite{brouwer}. Then \( A \), being the disjoint union of its connected components, is a \( \mathcal{R} \)-orientable closed manifold. Moreover, every component of \( A_i \) is a component of \( A \), and every component of \( A \) is a component of \( A_i \) for some \( i \). The result follows.
		\end{proof}
		
		\addcontentsline{toc}{section}{References} 
		\bibliography{Jennybib.bib, mybib.bib}{}
		\bibliographystyle{amsalpha} 

		\end{document}